\documentclass[12pt]{amsart} 

\usepackage[utf8]{inputenc} 
\pdfoutput=1

\usepackage{xy, graphicx, color,hyperref,verbatim,cleveref,tikz-cd}

\xyoption{all} 

\usepackage{amsfonts,array}

\usepackage{amssymb,amsmath,amsthm,bm,mathrsfs,enumerate,bbm,enumitem,abraces}
\usepackage[a4paper, total={6in, 8in}]{geometry}

\usepackage{fancyhdr}
\usepackage{mathtools,cite,mathdots}
\usepackage{tikz}
\usepackage{tikz-cd}
\usetikzlibrary{matrix}
\usepackage{siunitx}
\usepackage{booktabs}
\usepackage{blindtext}
\usepackage{booktabs} 
\usetikzlibrary{shapes.geometric}

\newcommand{\nc}{\newcommand}

\newcolumntype{P}[1]{>{\centering\arraybackslash}p{#1}}
\PassOptionsToPackage{linktocpage}{hyperref}
   \hypersetup{
           breaklinks=true,  
           colorlinks=true,  
           pdfusetitle=true, 
           citecolor={blue!80!black}
        }
\nc{\mc}{\mathcal}
\nc{\mb}{\mathbb}
\nc{\mf}{\mathfrak}
\nc{\ul}{\underline}
\nc{\ol}{\overline}
\nc{\dmo}{\DeclareMathOperator}
\nc{\R}{\mb R}
\dmo{\Spin}{Spin}
\dmo{\SO}{SO}
 \dmo{\pr}{pr}
 \dmo{\Sym}{Sym}
\dmo{\U}{U}
\dmo{\Hom}{Hom}
\dmo{\PGL}{PGL}
\dmo{\PSL}{PSL}
\dmo{\ortho}{orth}
\dmo{\sgn}{sgn}
\dmo{\dyn}{dyn}
\dmo{\trace}{Trace}
\dmo{\new}{new}
\dmo{\Ad}{Ad}
\dmo{\sym}{sym}
\dmo{\pal}{pal}
\dmo{\sd}{sd}
\dmo{\Ch}{Ch}
\dmo{\Span}{Span}
\dmo{\ord}{ord}
\dmo{\Or}{O}
\dmo{\core}{core}
\dmo{\quo}{quo}
\dmo{\Pin}{Pin}
\dmo{\Od}{Od}
\dmo{\EG}{EG}
\dmo{\BG}{BG}
\dmo{\ESO}{ESO}
\dmo{\BSO}{BSO}
\dmo{\BH}{BH}
\dmo{\EH}{EH}
\dmo{\Sgn}{Sgn}
\dmo{\irr}{Irr}
\dmo{\orb}{Orb}
\dmo{\rep}{Rep}
\dmo{\orep}{ORep}
\dmo{\odd}{odd}
\dmo{\dett}{det}
\dmo{\St}{St}
\dmo{\diag}{diag}
\dmo{\PP}{P}
\dmo{\hc}{H}

\newcommand{\as}{\alpha}

\newcommand{\f}{\mathbb{F}}
\newcommand{\z}{\mathbb{Z}}
\newcommand{\rr}{\mathbb{R}}
\newcommand{\cc}{\mathbb{C}}

\nc{\vt}{\vartheta}

\newcommand{\fe}{\mathfrak{e}}

\newcommand{\s}{\mathcal{S}}

\dmo{\rank}{rank}

\newcommand{\w}{\mathcal{W}}

 \newtheorem{thm}{Theorem}[section]
\newtheorem{c.intro}[thm]{Corollary}
\newtheorem{lemma}[thm]{Lemma}
\newtheorem{prop}[thm]{Proposition}

\newtheorem{cor}{Corollary}[thm]

\theoremstyle{definition}

\newtheorem{ex}[thm]{Example}
\theoremstyle{definition}

\theoremstyle{definition}

\dmo{\Mod}{mod}
\dmo{\res}{res}
 \dmo{\Sq}{Sq}
  \dmo{\Tr}{Tr}
 \dmo{\RO}{RO}
\dmo{\Sp}{Sp}
\dmo{\SL}{SL}
\dmo{\GL}{GL}
\dmo{\GSp}{GSp}
\nc{\la}{\lambda}
\nc{\eps}{\varepsilon}
\nc{\lip}{\langle}
 \nc{\rip}{\rangle}
\nc{\gm}{\gamma}
\nc{\beq}{\begin{equation*}}
\nc{\eeq}{\end{equation*}}
\dmo{\Perm}{Perm}
\dmo{\Res}{Res}
\dmo{\Ind}{Ind}
\dmo{\ind}{ind}
\dmo{\tr}{tr}
\dmo{\reg}{reg}
\dmo{\End}{End}
\dmo{\SW}{SW}
\dmo{\Syl}{Syl}
\MakeRobust\(
\MakeRobust\)

\title[Stiefel-Whitney Classes]{Stiefel-Whitney Classes for  Finite Symplectic  Groups }
\author{Neha Malik}
\author{Steven Spallone}

\address{Neha Malik, Chennai Mathematical Institute, Siruseri-603103, Tamil Nadu, India}
\email{51nehamalik94@gmail.com}
\address{Steven Spallone, Indian Institute of Science Education and Research, Pune-411008, Maharashtra, India}
\email{sspallone@gmail.com}

\keywords{Stiefel-Whitney classes, Symplectic Groups, Finite groups of Lie type, Weil Representations}
\subjclass{Primary 20G40, Secondary 55R40}

\begin{document}
\maketitle

\begin{abstract}
Let $q$ be an odd prime power, and $G=\Sp(2n,q)$ the finite symplectic group. 
We give an expression for the total Stiefel-Whitney Classes (SWCs) for orthogonal representations $\pi$ of $G$, in terms of character values of $\pi$ at elements of order $2$. We give ``universal formulas'' for the fourth and eighth SWCs. For $n=2$, we compute the subring of the mod $2$ cohomology generated by the SWCs $w_k(\pi)$.
\end{abstract}
\tableofcontents
\section{Introduction}
 
 Stiefel-Whitney Classes (SWCs) are interesting natural cohomological invariants of orthogonal representations. In this paper we present a formula for SWCs for the finite symplectic groups in odd characteristic.
 
This paper is part of a project to understand SWCs for finite groups of Lie type.   Let $q$ be an odd prime power throughout. A formula determining SWCs for $G=\GL(n,q)$ was discovered in \cite{GJgln},    for  $G=\SL(2,q)$ in  \cite{NSSL2} and for $G=\SL(2n+1,q)$ in \cite{MSSLn}. In this paper we find similar expressions when $G=\Sp(2n,q)$.
 
 Write $\hc^*(G)$ for the mod $2$ cohomology $\hc^*(G,\mb F_2)$, and $\hc^*_{\SW}(G)$ for the subalgebra of $\hc^*(G)$  generated by SWCs of orthogonal representations.  The   diagonal matrices in $G$ with eigenvalues $\pm1$ form a subgroup we denote by $Z_X$.  (It is the center of another subgroup $X$ which we will encounter later.)  The mod $2$ cohomology of $Z_X$ is a polynomial algebra:
$$\hc^*(Z_X)\cong \mathbb F_2[v_1,\hdots,v_n],$$
 where each $v_i$ is the first SWC of a certain linear character of $Z_X$; in particular it has degree 1. 
Let $\fe_i=v_i^4$, and let $\mc E_k$ be the $k$th elementary symmetric polynomial in the $\fe_i$.
 
 \begin{thm}\label{dtsym}
The restriction map $\hc^*_{\SW}(G) \to \hc^*(Z_X)$ is injective, and its image is contained in $\mathbb F_2[\mc E_1, \ldots, \mc E_n]$.
\end{thm}
 
Since no information is lost by restriction, we will express SWCs in terms of the $\mc E_k$. For $0 \leq i \leq n$, let $g_i \in Z_X$ with $-1$ having multiplicity $2i$ as an eigenvalue, and $1$ having multiplicity $2(n-i)$.  For example $g_0=\mathbbm1$, the identity matrix. Given a representation $\pi$, write $\chi_\pi$ for its character. Here is our ``universal'' formula for the 4th and 8th SWCs:
 
\begin{thm} \label{univ.intro}
If $\pi$ is an  orthogonal representation of $\Sp(2n,q)$, then
\beq
w_4(\pi)  = \frac{1}{8} \left( \deg \pi - \chi_\pi(g_1) \right)\mc E_1 \quad \quad \quad \text{     for all $n$,} \\
	\eeq
	 
	and
	\beq
	w_8(\pi)=r_1  \mc E_2 + \left( \binom{r_1}{2}+  \binom{r_2}{2} \right) \mc E_1^2 \quad \text{  for $n \geq 2$}.\\
 	\eeq
	Here
	\beq
	r_1 =\frac{1}{16}( \deg \pi-\chi_\pi(g_2))
	\eeq
	and
	\beq
	r_2=\frac{1}{16}( \deg \pi-2 \chi_\pi(g_1)+ \chi_\pi(g_2)).
	\eeq
	\end{thm}

   A general formula for the total SWCs $w(\pi)$ is found in Theorem \ref{SpnSWCs}, although it takes some computation  to extract individual SWCs $w_k(\pi)$ from this. For example: 

\begin{thm} \label{intro.thm.sp4}
 The total SWC of an orthogonal representation $\pi$  of $\Sp(4,q)$ is
$$w(\pi)=((1+\fe_1)(1+\fe_2))^{r_\pi}(1+\fe_1+\fe_2)^{s_\pi},$$
where \begin{align*} r_\pi&=\frac{1}{16}\Big(\chi_\pi(\mathbbm{1})-\chi_\pi(-\mathbbm{1})\Big) \quad \quad \text{ and }\\
s_\pi&=\frac{1}{16}\Big(\chi_\pi(\mathbbm{1})+\chi_\pi(-\mathbbm{1})-2\chi_\pi(g_1)\Big).
\end{align*}
\end{thm}

We use this theorem to compute:
\begin{cor}  \label{HSW}
The subalgebra
$$\hc^*_{\SW}(\Sp(4,q))\cong \mathbb F_2[\fe_1+\fe_2,\fe_1\fe_2].$$
\end{cor}

One can in principle produce universal formulas for all $w_k$, akin to those of Theorem \ref{univ.intro} by the following.
Let   $G_n=\Sp(2n,q)$. 

\begin{thm} \label{intro.thm.univers}
Let $m\geq n$. The restriction map $\iota_n^*: \hc^i(G_m)\to \hc^i(G_{n-1})$ is injective for   $i<4n-1$.

\end{thm}
(Compare \cite[Corollary 6.7]{NSn} and  \cite[Theorems 3, 8]{GJgln}.)

 A representation $\pi$ of a group $G$ has a total Chern class $c(\pi) \in \hc^{*}(G,\z)$. Under the coefficient map to $\hc^{*}(G,\mathbb F_2)$, these map to $w(\pi \oplus \pi^\vee)$. Hence this ``mod $2$ Chern class of $\pi$'' is computable for $G=\Sp(2n,q)$ by our formulas. 

The paper is laid out as follows. Preliminaries are reviewed and developed in Section \ref{section2}.
Section \ref{ch5} contains the heart of the paper; we prove Theorem \ref{dtsym}, and give the product formula for the total SWC,
as Theorem  \ref{SpnSWCs}. Some simplifications come when the representation is irreducible, by a formula of Gow. 
We develop our formulas for $\Sp(4,q)$ and $\Sp(8,q)$ in Section \ref{examples.section}, and prove Theorem \ref{intro.thm.sp4} and Corollary \ref{HSW}.  We also illustrate application of our formula by computing the mod $2$ Chern class of the Weil representations.
Finally, in Section \ref{universal.section}  we establish Theorem \ref{intro.thm.univers}, and deduce Theorem \ref{univ.intro} from this.
 Certain technical arguments belonging to vector bundle theory are sketched in the Appendix.

\section{Notations and Preliminaries} \label{section2}

As this paper is a continuation of \cite{NSSL2}, we use the same notations and conventions, which we now review.

\subsection{Representations}

Let $G$ be a finite group. All the representations $(\pi,V)$ considered in this paper (before the Appendix) are complex finite dimensional.  Let $\irr(G)$ be the set of isomorphism classes of irreducible representations of $G$. Write $(\pi^\vee,V^\vee)$ for the dual representation. If $H$ is a subgroup of $G$, write $\res^G_H\pi$ or $\pi|_H$ for the restriction of $\pi$ to $H$.  A \emph{linear character} $\chi$ of $G$ is a degree 1 representation. We call $\chi$ \textit{quadratic} when $\chi^2=1$.

We say $\pi$ is \emph{orthogonal} (resp., \emph{symplectic}), provided there exists a non-degenerate $G$-invariant symmetric (resp., antisymmetric) bilinear form $B : V \times V \to \cc$. 
When $\pi$ is self-dual and irreducible, it is either orthogonal or symplectic. In this
case, the Frobenius-Schur Indicator $\varepsilon(\pi)$ is a sign defined as $1$ when $\pi$ is orthogonal,
and $-1$ when $\pi$ is symplectic. Whereas it is $0$, when $\pi$ is not self-dual.

One can symmetrize a general $(\pi,V)$ by defining $S(\pi):=\pi\oplus\pi^\vee$ on the vector space $V\oplus V^\vee$. Under the symmetric  $G$-invariant bilinear map $B$ on $V\oplus V^\vee$ as $B((v,\as),(w,\beta))= \langle \alpha,w \rangle+ \langle \beta,v \rangle$, the representation $S(\pi)$ is orthogonal. We call $S(\pi)$ the \textit{symmetrization} of $\pi$. (In \cite[Definition 4.3 ]{shl} it is called the \emph{hyperbolic space} on $V$.)

Every orthogonal  representation  $\Pi$ of $G$ can be decomposed as 
\begin{equation}\label{decorth}\Pi\cong  \bigoplus_i \pi_i \oplus\bigoplus_j S(\varphi_j), \end{equation}
such that each $\pi_i$ is irreducible orthogonal and $\varphi_j$ are irreducible non-orthogonal representations of $G$.

A   representation $\pi$ of $G$ is said to be an \textit{orthogonally irreducible} representation (OIR), provided $\pi$ is orthogonal, and can not be decomposed into a direct sum of orthogonal representations. An irreducible representation $\pi$ is orthogonally irreducible if and only if $\pi$ is orthogonal. Moreover, for $\varphi$ irreducible and non-orthogonal, its symmetrization $S(\varphi)$ is an OIR.

\subsection{Detection} \label{nlizer}

 As in the earlier work, we make use of \textit{detection}. To recall this notion, let $i:H\hookrightarrow G$ be a subgroup.  We say $H$ \textit{detects the mod $2$ cohomology of} $G$, when the restriction map
$$i^*:\hc^*(G)\rightarrow \hc^*(H)$$
 is injective, and we say $H$ \emph{detects SWCs} of $G$ when the restriction of $i^*$ to $\hc^*_{\SW}(G)$ is injective. Often the cohomology of $H$ admits an easy description, e.g., when it is polynomial. Then it is convenient to give our formulas there, since no information is lost.
 Let $N_G(H)$ be the normalizer of $H$ in $G$, which acts on $H$ by conjugation. This induces an action of $N_G(H)$ on $\hc^*(H)$, and generally the image of $i^*$ is contained in the subalgebra 
  $\hc^*(H)^{N_G(H)}$ fixed under this action.

\subsection{Characteristic Classes}
Let $\pi$ be an orthogonal representation of degree $d$. Associated to $\pi$ are cohomological invariants
$$w_i(\pi)\in \hc^i(G)\quad;\quad i=0,1,2,\hdots,d$$
known as the $i$th \emph{Stiefel-Whitney Class} (SWC) of $\pi$. Their sum $w(\pi)=w_0(\pi)+w_1(\pi)+\hdots$ is called the \emph{total SWC} of $\pi$. We refer the reader to \cite[Section 2.3]{NSSL2} for detailed description.

Also, associated to a complex representation $\pi$ of $G$ are cohomology classes $c_i(\pi)\in \hc^{2i}(G,\z),$
 called \textit{Chern classes} (CCs). Their sum
 \beq
c(\pi)=c_0(\pi)+c_1(\pi)+c_2(\pi)+\hdots\in \hc^*(G,\z)
\eeq
is called the \textit{total Chern class} of $\pi$. We have  \begin{equation}\label{chswc}w(S(\pi))=\kappa(c(\pi)),\end{equation}
 where $\kappa : \hc^*(G,\z)\rightarrow \hc^*(G,\mathbb F_2)$ is the  \emph{coefficient homomorphism} of cohomology. (See \cite[Lemma 1]{ganguly2022stiefel}, based on \cite[Problem 14-B]{milnor} for proof.) So we interpret $w(S(\pi))$  as the ``mod $2$ Chern class'' of $\pi$.

For $n,i\geq 0$, there are additive homomorphisms on cohomology, called \textit{Steenrod Squares},
 $$\Sq^i:\hc^n(G)\rightarrow \hc^{n+i}(G).$$
These operations are \textit{functorial}, meaning for a group homomorphism  $\varphi: G_1\to G_2$, we have $$\varphi^*(\Sq^iy)=\Sq^i(\varphi^*(y)) \text{ for all }y\in \hc^i(G_2).$$
   They satisfy $\Sq^i(x)=x\cup x$ for $i=\deg(x)$, and $\Sq^i(x)=0$ for $i>\deg(x)$. There is the \textit{Cartan Formula}: For $x,y\in \hc^*(G)$, \begin{equation}\label{cartan}
   \Sq^n(x\cup y)=\sum_{i+j=n}(\Sq^ix)\cup (\Sq^jy).
   \end{equation}

The well-known \textit{Wu's formula} states:
\begin{prop}[\cite{may}, Chapter 23, Section 6]\label{Wu}
Let $\pi$ be an orthogonal representation of $G$. The cohomology class $\Sq^i(w_j(\pi))$ can be expressed as a polynomial in $w_1(\pi),\hdots,w_{i+j}(\pi)$:
$$\Sq^i(w_j(\pi))=\sum_{t=0}^i{j+t-i-1\choose t}w_{i-t}(\pi)w_{j+t}(\pi).$$
\end{prop}

 \begin{cor} Suppose $w_1(\pi)=w_2(\pi)=0$. Then if $w_i(\pi) \neq 0$, then $i$ is a multiple of $4$.
\end{cor}

\begin{proof} This is clear.
\end{proof}
\subsection{Symmetric Functions}
The elementary symmetric functions $\mc{E}_k(\mathbf x)$ over $\mathbb F_2$ in variables $x_1,x_2,\hdots,x_n$ are defined for $k\leq n$ as,
$$\mc{E}_k(\mathbf x):=\sum_{1\leq i_1<\hdots<i_k\leq n}x_{i_1}x_{i_2}\hdots x_{i_k}\in \mathbb F_2[x_1,\hdots,x_n].$$
Whereas for $k>n$, these are defined to be $0$.
For example, $\mc E_1(x_1,x_2)=x_1+x_2$ and $\mc E_2(x_1,x_2)=x_1x_2$. Of course,  
$$\prod_{i=1}^n (1+x_i)=1+\mc E_1(\mathbf x)+\hdots+\mc E_n(\mathbf x).$$

\subsection{Cyclic Groups}\label{cy}
Let $k$ be even, and $G=C_k$ the cyclic group of order $k$. Let $g$ be a generator of $G$. We write $g^{k/2}=-1$, the unique order $2$ element of $G$. We say a linear character $\chi: G \to \cc$ is \emph{odd}, when $\chi(-1)=-1$, and \emph{even}, when $\chi(-1)=1$.
Let `sgn' denote the unique non-trivial quadratic character of $G$.

It is known \cite{KT} that  for  $k\equiv2 \pmod 4$, 
$\hc^*(C_k)=\mathbb F_2[v] 
$ where $v=w_1(\sgn)$. In this case, (see \cite[Lemma 2.5]{NSSL2} for instance) for a representation $\pi$, \beq
w(\pi)=(1+v)^{b_\pi}
\eeq
with $b_\pi=\cfrac{1}{2} \left( \deg \pi-\chi_\pi(-1) \right)$.\\

Let $C_2^n$ be the $n$-fold product of $C_2$, with projection maps $\pr_i: C_2^n\to C_2$ for $i=1, \hdots,n$. By K\"{u}nneth, we have 
\begin{equation}\label{Cnk}\hc^*(C_2^n)=\mathbb F_2[v_1,\dots,v_n], \end{equation}
where we put $v_i=w_1(\sgn\circ\pr_i)$ for $1\leq i\leq n$.

\subsection{OIRs of a Direct Product of Groups}\label{OIRD}
Let $G_1,\hdots,G_n$ be finite groups and $G=G_1\times\cdots\times G_n$ be their direct product.
Given $G_i$-representations $(\pi_i,V_i)$, one can form their \textit{external tensor product} $ \pi= \pi_1\boxtimes\cdots\boxtimes\pi_n$ from the action of the product group $G$ on the tensor space $V_1\otimes\cdots\otimes V_n$ as,
$$(g_1,\hdots,g_n)(v_1\otimes\cdots\otimes v_n)=g_1v_1\otimes\cdots\otimes g_nv_n$$
for $(g_1,\hdots,g_n)\in G$ and $v_1\otimes\cdots\otimes v_n\in V_1\otimes\cdots\otimes V_n.$  (See \cite[Exercise 2.36, page 24]{fulton}.) 
 If each $\pi_i$ is irreducible, then so is
$\pi$. All irreducible representations of $G$ decompose in this manner.
The Frobenius-Schur indicator of $\pi$ is given by  
\begin{equation}\label{fsiprod}
\begin{split}
\varepsilon(\pi) 
&=\varepsilon(\pi_1)\varepsilon(\pi_2)\hdots\varepsilon(\pi_n).
\end{split}
\end{equation}
Write $F(\pi)$ for the multiset $\{\pi_1,\pi_2,\hdots,\pi_n\}$. 
Then from \eqref{fsiprod}, we can describe the OIRs of $G$ in terms of irreducible representations of $G_i$ as follows: 
\begin{enumerate}
\item \emph{Irreducible orthogonal representations} of the form $\pi=\pi_1\boxtimes\cdots\boxtimes\pi_n,$ where
$\pi_i\in \irr(G_i)$ are self-dual  for each $i$ and an even number of representations in $F(\pi)$ are symplectic. 

\item \emph{Symmetrization of irreducible non-orthogonal representations}
 \newline $\pi=S(\varphi)=S(\varphi_1\boxtimes\cdots\boxtimes\varphi_n),$ where $\varphi_i\in \irr(G_i)$ for $i=1,\hdots, n$ and exactly one of the following holds:
\begin{enumerate} \item At least one of $\varphi_i$ is not self-dual.
	\item  Each $\varphi_i$ is self-dual and there is an odd number of symplectic representations
 in $F(\varphi)$.
\end{enumerate}
\end{enumerate}

\subsection{Quaternion Group}\label{qg}

Let $Q$ be the quaternion group of order 8. There are four linear characters of $Q$, which we denote by $1$, $\chi_1,\chi_2, \chi_1\otimes\chi_2$. Each one is quadratic.  The group also possesses a unique   irreducible representation  $\rho$ of degree $2$.  It is symplectic.

Let $Q^n$ be the $n$-fold product of $Q$. From Section \ref{OIRD} above, any irreducible representation $\pi$ of $Q^n$ has  the form
$\pi\cong\pi_1\boxtimes\cdots\boxtimes\pi_n$
where each $\pi_i\in \irr(Q)$. Since every representation of $Q$ is self-dual, the same is true for $Q^n$.

Put $I_\pi=\{i:\pi_i\cong \rho\}$, and $r(\pi)=|I_\pi|$.
From \eqref{fsiprod} we obtain:
\begin{lemma}\label{orev}
An irreducible representation $\pi$ of $Q^n$ is orthogonal if and only if $r(\pi)$ is even. 
\end{lemma}

Let $Z$ be the center of $Q$, which is $\{\pm 1\}$. For a linear character $\theta$ of $Z^n$, write $[\theta, \pi]$ for the multiplicity of $\theta$ in $\pi|_{Z^n}$.
 
\begin{lemma}\label{ORQn}
Let $\pi$ be an orthogonal representation of $Q^n$. Let $\theta$ be a non-trivial linear character of $Z^n$. Then
$[\theta, \pi]$ is a multiple of $4$.
\end{lemma}

\begin{proof}
 We may assume that $\pi$ is an OIR. Let $r=r(\pi)$.
  
 If $r=0$, then $\pi|_Z$ is trivial, since $Z$ is in the kernel of all linear characters of $Q$.  
So take $r>0$. 
Define a linear character $\theta_\pi$ of $Z^n$ by $\theta_{\pi}=\boxtimes_i \theta_i$, with
\beq
\theta_i=\begin{cases} \sgn, & i \in I_\pi \\
1 & i \notin I_\pi. \\
\end{cases}
\eeq
Then  $\pi|_{Z^n}\cong 2^r\theta_\pi$,
since $\rho|_Z =\sgn \oplus \sgn$. 
If $\pi$ is irreducible orthogonal, then $r$ is even by Lemma \ref{orev} and therefore $4$ divides $[\theta_\pi,\pi]$, and other $[\theta,\pi]=0$.

Otherwise $\pi=S(\phi)$ for an irreducible symplectic representation $\phi$.
In this case,  $$[\theta_\pi,S(\phi)]=2^r+2^r=2^{r+1},$$ which is again divisible by $4$ for $r>0$.
\end{proof}

\subsection{$S_n$-invariant Representations of Elementary Abelian $2$-groups}\label{eab}

Let $E$ be an elementary abelian $2$-group of rank $n$. View $E$ as an $\mb F_2$-vector space, say with basis $e_1,\hdots,e_n$. For $e\in E,$ put $|e|=\#\{i:c_i=1\}$ when $e=\sum\limits_{i=1}^r c_ie_i$ with $c_{i}\in \mathbb F_2$. Put $\mc O_k=\{e\in E:|e|=k\}.$

Let $E^\vee=\Hom(E,\mathbb{F}_2)$ and consider the basis $v_1, \hdots, v_n$ of $E^\vee$ dual to the $e_i$.  For $v\in E^\vee,$ put $|v|=\#\{i:c_i=1\}$ when $v=\sum\limits_{i=1}^r c_iv_i$ with $c_{i}\in \mathbb F_2$. The representation $$\sigma_k=\bigoplus_{|v|=k} v \quad;\quad k=0,1,\hdots,n$$
 is $S_n$-invariant, of degree $\binom{n}{k}$. 
  For a polynomial $f$, write $[f]_i$ for the coefficient of the degree $i$ term of $f$; in other words so that $f(x)=\sum_i [f]_i x^i$.

\begin{lemma}[\cite{GJgln}, Proposition 2]\label{chisig}
For $e \in \mc O_k$, we have 
\beq
\chi_{\sigma_i}(e)= \left[ (1-x)^k (1+x)^{n-k} \right]_i.
\eeq
\end{lemma}

Moreover, any $S_n$-invariant representation $\sigma$ of $E$  is a direct sum of $\sigma_k$'s:
\begin{equation}\label{pisig}\sigma=\bigoplus_{k=0}^n m_{k}(\sigma)\sigma_k\end{equation}
where $m_k(\sigma)$ are certain non-negative integers.
 The coefficients $m_{k}(\sigma)$  can be expressed in terms of character values of $\sigma$ at the elements of $E$.

 For $0 \leq k \leq n$, put
$\vartheta_k=e_1+\hdots+e_k$. (In particular, $\vartheta_0=0$.)
Then,
\begin{equation}\label{mk}m_k(\sigma)=\frac{1}{2^{n}}\sum\limits_{i=0}^n\chi_{\sigma_i}(\vartheta_k)\chi_\sigma(\vartheta_i).\end{equation}
From Lemma \ref{chisig}, the character value $\chi_{\sigma_i}(\vartheta_k)$ is the coefficient of $x^i$ in the expression $(1-x)^k(1+x)^{n-k}$. Following \cite[Theorem 7]{GJgln} we have 
\begin{equation} \label{big.old.prod}
w(\sigma) = \prod\limits_{k=1}^n\Big(\prod\limits_{|v|=k }(1+ v)\Big)^{m_k(\sigma)}.
\end{equation}

 For later use, we state the following lemma; its proof is immediate.
 
 \begin{lemma} \label{msdeed}  Let  $\sigma$ be an $S_n$-invariant representation of $E$ such that 
$\chi_\sigma(\vartheta_i)=\chi_\sigma(\vartheta_{n-i})$ for all  $0\leq i\leq n$. Then, we have
\beq
m_k(\sigma)=\begin{cases}
0 & \text{ when } k \text{ is odd}\\[5pt]
\cfrac{1}{2^{n-1}}\sum\limits_{i=0}^{\frac {n-1}2}\chi_{\sigma_i}(\vartheta_k)\chi_\sigma(\vartheta_i) & \text{ when } k \text{ is even, } n \text{ is odd}\\[15pt]

\cfrac{1}{2^{n-1}}\sum\limits_{i=0}^{\frac {n-2}2}\chi_{\sigma_i}(\vartheta_k)\chi_\sigma(\vartheta_i)+\cfrac{1}{2^n}\big(\chi_{\sigma_{\frac n2}}(\vartheta_k)\chi_\sigma(\vartheta_{\frac n2})\big) & \text{ when } k,n \text{ both  are even}
\end{cases}
\eeq
 for $1\leq k\leq n.$
\end{lemma}

\subsection{Dickson Invariants}\label{dickin}

With $E$ as the elementary abelian $2$-group above, write  $\Sym(E^\vee)$ for the symmetric algebra of $E^\vee$ over $\mb F_2$.
 In this algebra, we have an element
\beq
\mc D(E)= \prod_{v \in E^\vee}(1+v)=1+\sum_{i=1}^n d_{i}(E),
\eeq
for certain $\GL(E)$-invariant polynomials $d_i$ of degree $2^n-2^{n-i}$, known as \emph{Dickson invariants}. 
 Note that $w(\reg(E))=\mathcal{D}(E)$, where $\reg(E)$ is the regular representation (as mentioned in \cite{wilk}).
 
Certain factors of $\mc D(E)$ enter into our work, which we explain here. Consider the basis
$
{\bm v}=\{v_1, \ldots, v_n\}$
of $E^\vee$. Identifying $E^\vee$ with $\hc^1(E)$, we may write
$$\hc^*(E)=\Sym(E^\vee)\cong \mathbb F_2[v_1,\hdots,v_n].$$
  
Put
\beq
\mc D^{[k]}({\bm v})= \prod_{|v|=k} (1+v) \in \Sym(E^\vee),
\eeq
 so that $\mc D(E)=\prod_{k} \mc D^{[k]}({\bm v})$. Clearly, $\mc D^{[k]}({\bm v})$ is a symmetric polynomial. We have
\beq
\mc D^{[n]}({\bm v})=1+\mc E_1 \quad \quad \text{ and } \quad \quad \mc D^{[1]}({\bm v})=1+\mc E_1+ \mc E_2+ \ldots + \mc E_n,
\eeq

  For $n=4$, we have
\beq
\begin{split}
\mc D^{[2]}({\bm v})&=1+ \mc E_1+ \mc E_1^2 + \mc E_1^3+( \mc E_2^2+\mc E_1\mc E_3)+(\mc E_1\mc E_2^2+\mc E_1^2\mc E_3)+(\mc E_1\mc E_2\mc E_3+\mc E_3^2+\mc E_1^2\mc E_4) \text{ and }\\
\mc D^{[3]}({\bm v})&=1+ \mc E_1+ (\mc E_1^2 + \mc E_2)+( \mc E_3+ \mc E_1^3)+( \mc E_1^2\mc E_2+ \mc E_1 \mc E_3+\mc E_4).
\end{split}
\eeq

Thus we can rewrite \eqref{big.old.prod} as:
 \begin{prop}\label{wsigma}
 For an $S_n$-invariant representation $\sigma $ of $E$, we have
\beq
w(\sigma)=  \prod\limits_{k=1}^n\mc D^{[k]}({\bm v})^{m_{k}(\sigma)}
\eeq
where $m_k(\sigma)$ is given by \eqref{mk}.
\end{prop}
 
 \section{Symplectic Groups}\label{ch5}
  
  \subsection{Subgroups of $G$} \label{subgroup.sec}
  Let $\mathcal{S}=\SL(2,q)$. Let $Z=\{ \pm {1}\}$ be the center of $\mc S$.
  Set 
\begin{small}$$J=\begin{pmatrix}
0&&0&&\hdots&&0&&1\\
0&&0&&\hdots&&-1&&0\\
\vdots&&\vdots&&\iddots&&\vdots&&\vdots\\
0&&1&&\hdots&&0&&0\\
-1&&0&&\hdots&&0&&0\\
\end{pmatrix},$$
\end{small}
and put
$$G=\Sp(2n,q)=\{A\in \GL(2n,q): A^tJA=J\}.$$ 
 Let $X$ be the subgroup of matrices in $G$ of the form
\beq
A= \begin{pmatrix}
a_{n}&0&0& \hdots&\hdots&0&0&b_{n}\\
0&\ddots&& &&&\iddots&0\\
&&a_2&&&b_2&&\\
\vdots&&&a_1 & b_1 &&&\vdots\\
\vdots&&&c_1 &d_1&&&\vdots\\
&&c_2&&&d_2&&\\
0&\iddots&&&&&\ddots&0\\
c_n&0&0&  \hdots &\hdots&0&0&d_n\\
\end{pmatrix},
\eeq
meaning the nonzero entries of $A$ lie either on the diagonal or the antidiagonal. 
Let $Z_X$ be the center of $X$; it is the subgroup of diagonal matrices in $G$ which have $1$ or $-1$ on the diagonal.

Note that each $A_i: =\begin{pmatrix} a_i &b_i \\ c_i &d_i  \end{pmatrix} \in \mc S$, and
$A \mapsto (A_1, \ldots, A_n)$ maps
$X$ isomorphically to the $n$-fold product $\mc S^n$. It also maps $Z_X$ to $Z^n$.
 
Let $M$ be the subgroup of diagonal $n \times n$ block matrices of $G$; then 
\beq
\begin{pmatrix} A &0 \\ 0 &B  \end{pmatrix} \mapsto A
\eeq
maps $M$ isomorphically to $\GL(n,q)$. The subgroup of $n \times n$ permutation matrices in $\GL(n,q)$ thus gives a copy of $S_n$ in $G$; this subgroup normalizes $X$, and acts by conjugation on $Z_X < X$ by the given permutation.

\subsection{Detecting SWCs}

In this Section, we prove Theorem \ref{dtsym}. Our strategy is to combine the facts, reviewed below, that $X$ is a detecting subgroup of $G$, and that $Z$ detects the SWCs of $\mc S$. Along the way, we must determine $\hc^*_{\SW}(\mc S^n)$.

  It is well known (see \cite[Chapter VI, Sec. 5]{FP} for instance) that the mod $2$ cohomology ring of $\mc S$ is  
$$\hc^*(\mathcal S)\cong \mathbb F_2[\mathfrak{e}]\otimes \mathbb F_2[\mathfrak{b}]/\langle \mathfrak{b}^2\rangle$$
with $\deg(\mathfrak{b})=3$, $\deg(\mathfrak{e})=4$. 

In \cite[Corollary 4.8]{NSSL2}, we specified an orthogonal representation $\eta$ of $\mc S$ with the property that $w_4(\eta)=\mf e$.
We also studied the restriction map   $\hc^*(\mc S) \to H^*(Z)$ and computed that $\mathfrak e$ maps to $v^4 \in \hc^4(Z)$.  This gives an isomorphism 
\begin{equation} \label{alphabet}
\hc^4(\mc S) \overset{\sim}{\to} \hc^4(Z), 
\end{equation}
and leads to:
\begin{thm}[\cite{NSSL2}] \label{prop} The center $Z$ detects SWCs of $\mc S$. We have $\hc ^*_{\SW}(\s)=\mathbb F_2[\mf e]$.
\end{thm}

 From above, consider the subgroup $X$ of $G=\Sp(2n,q)$, isomorphic to $\mathcal{S}^n$.
 There are projections
$$\text{pr}_j:\mathcal{S}^n\to \mathcal{S}\quad \text{ ; }\quad1\leq j\leq n,$$ and by  K\"{u}nneth we have
\begin{equation}\label{kf}\hc^*(X)\cong \hc^*(\mc S^n)\cong \mathbb F_2[\mathfrak{e}_1,\hdots,\mathfrak{e}_n]\otimes_{\mathbb F_2} \mathbb F_2[\mathfrak{b}_1,\hdots,\mathfrak{b}_n]/( \mathfrak{b}_1^2, \ldots, \mathfrak{b}_n^2),\end{equation}
where $\mathfrak{e}_j=\text{pr}^*_j(\mathfrak{e})$ and $\mathfrak{b}_j=\text{pr}^*_j(\mathfrak{b})$.   Note that $\mathfrak{e}_j=w_4(\eta_j)$ with $\eta_j=\eta\circ\pr_j$ for each $j$.  

\begin{lemma} [\cite{Milgram}, Chapter VII, Lemma 6.2]\label{spsl2}
The subgroup $X$ detects the mod $2$ cohomology of $G$. 
\end{lemma}
From before, the subgroup $Z_X$ of $G$ is normalized by $S_n$ and its mod $2$ cohomology ring is
$$\hc^*(Z_X)\cong \hc^*(C_2^n)\cong \mathbb F_2[v_1,\hdots,v_n]$$
by \eqref{Cnk}.

Consider the case $n=2$. Then by   \eqref{kf}  we have $ \hc^{k}(\mathcal{S}\times \mathcal{S})=0$ when $k \equiv 1 \pmod 4$, and
$ \hc^{k}(\mathcal{S}\times \mathcal{S}) \subseteq  \mathbb F_2[\mathfrak{e}_1,\mathfrak{e}_2]$ when $k \equiv 0 \pmod 4$. We have the following lemma: 
\begin{lemma}\label{l1}
The Steenrod square $\Sq^2:  \hc^{k}(\mathcal{S}\times \mathcal{S}) \to  \hc^{k+2}(\mathcal{S}\times \mathcal{S}) $ is trivial when $k$ is a multiple of $4$. 
\end{lemma}

\begin{proof} 
By the above, it is enough to show that $\Sq^2$ vanishes on $\mathbb F_2[\mathfrak{e}_1,\mathfrak{e}_2]$.

Since $\hc^5(\mc S)=\hc^6(\mc S)=\{ 0 \}$, we necessarily have $\Sq^1(\fe)=\Sq^2(\fe)=0$. By naturality of $\Sq$ we also have
$\Sq^1(\fe_i)=\Sq^2(\fe_i)=0$ for $i=1,2$. By Cartan's formula \eqref{cartan},  we deduce $\Sq^2(\mathfrak{e}_1^s\mathfrak{e}_2^t)=0$ for all  nonnegative $s,t$.
The conclusion follows. 
\end{proof}

\begin{prop}\label{prop1}
 We have
$$\hc^*_{\SW}(\s\times \s)=\mathbb F_2[\mathfrak{e}_1,\mathfrak{e}_2].$$
\end{prop}
\begin{proof} 
Since $w_4(\eta_i)=\mathfrak{e}_i$ for $i=1,2$,  the right side is contained in the left side.

Now let $\pi$ be an orthogonal representation of $\mathcal S \times \mathcal S$. By \eqref{kf} again it is enough to  see that $w_k(\pi)=0$ whenever $k\equiv2,3$ (mod $4$).

For $k\equiv 2$ (mod $4$), we use Wu's formula (Proposition \ref{Wu}) with $i=2$ and $j=k-2$, giving
\begingroup
\addtolength{\jot}{0.5em}
\begin{equation*}\label{wu1}
\begin{split}
\Sq^2(w_{k-2}(\pi))&\begin{small}={k-5\choose0}w_2(\pi) w_{k-2}(\pi)+{k-4\choose1}w_1(\pi) w_{k-1}(\pi)+{k-3\choose2}w_0(\pi) w_{k}(\pi)\end{small}\\
&=\frac{(k-3)(k-4)}{2}w_{k}(\pi).
\end{split}
\end{equation*}
\endgroup
The second equality holds because $\hc^i(\s\times \s)=\{0\}$ for $i=1,2$. Also $\frac{(k-3)(k-4)}{2}$ is odd for $k\equiv 2$ (mod $4$), and $\Sq^2(w_{k-2}(\pi))=0$ by Lemma \ref{l1}. Hence $w_k(\pi)=0$.

A similar argument works for $k\equiv 3$ (mod $4$): applying  Wu's formula with $i=1$ and $j=k-1$ shows that $w_{k}(\pi)=0$.
\end{proof}

\begin{thm}\label{lc}
We have
$$\hc^*_{\SW}(\s^n)=\mathbb F_2[\mathfrak{e}_1,\ldots,\mathfrak{e}_n].$$
\end{thm}
\begin{proof}
Again we have $\fe_i=w_4(\eta_i)\in \hc^*_{\SW}(\mc S^n)$, so the left hand side contains the right hand side.
For the other direction, it is enough to show that the SWCs of an OIR of $\s^n$ lie in  $\mathbb F_2[\mathfrak{e}_1,\ldots,\mathfrak{e}_n]$. 
  
 First we treat the case of ``decomposable'' orthogonal representations. For  $0 \leq r \leq n$ with $r$ even, let $\varphi_1, \ldots, \varphi_r$ be symplectic representations of $\s$, and 
 $\pi_1, \ldots, \pi_{s}$ orthogonal representations of $\s$, with $r+s=n$. Let
 \beq
 \Pi=\varphi_1\boxtimes\cdots \boxtimes \varphi_{r} \boxtimes \pi_1\boxtimes \cdots \boxtimes \pi_{s}
 \eeq
 be their external product. Then $\Pi$ is an orthogonal representation of $\s^n$, as in Section \ref{OIRD}. 
  Note that each of
$$\varphi_1\boxtimes\varphi_2\:,\: \hdots , \: \varphi_{r-1}\boxtimes\varphi_{r}$$
 is an orthogonal representation of $\s\times \s$. By Theorem \ref{prop}, Proposition \ref{prop1} and Corollary \ref{tengen}, we deduce
$w(\Pi)\in \mathbb F_2[\mathfrak{e}_1,\ldots,\mathfrak{e}_n]$.

 Next we treat symmetrizations of decomposable representations.
For representations $\pi_1,\hdots,\pi_n$  of $\mc S$, put $\Pi=\pi_1\boxtimes\cdots\boxtimes\pi_n$.
Then, $w(S(\Pi))=\kappa(c(\Pi))$. From Corollary \ref{tengen}, $c(\Pi)$ lies in the subalgebra of $\hc^*(\mathcal S^n,\z)$ generated by the $p_j^*c_i(\pi_j)$. Apply the homomorphism $\kappa$; then Theorem \ref{prop}   gives
 $$w(S(\Pi))\in\mathbb F_2[p_j^*w_{2i}(S(\pi_j))]\subseteq\mathbb F_2[\mathfrak{e}_1,\ldots,\mathfrak{e}_n].$$
 By the classification in Section \ref{OIRD}, up to the permutation action of $S_n$ on $\s^n$,
 each OIR takes one of the forms treated above.  Since $\mathbb F_2[\mathfrak{e}_1,\ldots,\mathfrak{e}_n]$ is stable under this action,   we are done.
\end{proof}

Now we are ready to prove our main detection theorem.
\begin{proof}[Proof of Theorem \ref{dtsym}]
From Equation \eqref{alphabet}, the restriction map $\hc_{\SW}^*(\mc S^n) \to \hc^*(Z^n)$ maps $\mf e_i$ to $v_i^4$; in our coordinates this gives an isomorphism $$\mathbb F_2[\mathfrak{e}_1,\ldots,\mathfrak{e}_n]  \cong \mathbb F_2[v_1^4,\hdots,v_n^4].$$

Now by Lemma \ref{spsl2}, we have 
\beq
\begin{split}
\hc^*_{\SW}(G) &\hookrightarrow \hc^*_{\SW}(X) \\
& \cong \hc^*_{\SW}(\s^n) \\
&=\mathbb F_2[\mathfrak{e}_1,\ldots,\mathfrak{e}_n] \text{ by Theorem \ref{lc}}\\
&\cong\mathbb F_2[v_1^4,\hdots,v_n^4]. \\
\end{split}
\eeq
The image lands in the $S_n$-invariant subalgebra, as in Section \ref{nlizer}.
\end{proof}

\subsection{Product Formula for Total SWCs }
Let $G=\Sp(2n,q)$, and $\pi$ be an orthogonal representation of $G$. To find $w(\pi)$, we work with
$$w(\pi|_{Z_X})\in \mathbb F_2[v_1^4,\hdots,v_n^4]^{S_n}$$ 
\noindent due to Theorem \ref{dtsym}. Since $\pi$ is $S_n$-invariant, we can write $\pi|_{Z_X}\cong \bigoplus\limits_{k=0}^nm_k(\pi)\sigma_k$ (with $\sigma_k$ as in Section \ref{eab}) and its total SWC is described by Proposition \ref{wsigma}. But we can say more about the coefficients $m_k(\pi)$ appearing in $w(\pi|_{Z_X})$ because $\pi|_{Z_X}$ is coming from a representation of the bigger group $G$. 

Let $2^{a+1}$ be the highest power of $2$ which divides $|\mc S|$. Gorenstein in the proof of \cite[Theorem 8.3, Chapter 2]{goren} specifies elements $x,y \in \mc S$ which generate a $2$-Sylow subgroup of $\mc S$. (It is generalized quaternion of order $2^{a+1}$.) One checks that the subgroup $Q$ generated by $y$ and $x^{2^{a-2}}$ is the quaternion group of order $8$. As $X$ is isomorphic to $\mc S^n$, it correspondingly has a subgroup  $Q_X$ isomorphic to $Q^n$, and containing $Z_X$. We thus have a chain of subgroups $Z_X \leq Q_X \leq X \leq G$. Since $Z_X$ detects SWCs of $G$, we   infer that $Q_X$ also detects the SWCs of $G$.

Let $\pi$ be an orthogonal representation of $G$. Clearly $\res^G_{Z_X}\pi=\res^{Q_X}_{Z_X}\res^G_{Q_X}\pi$ and is $S_n$-invariant. Now, Lemma \ref{ORQn} gives all $m_k(\pi)$ are divisible by $4$.

We can now obtain $w(\pi)$ as its image in $\hc^*(X)$ by identifying $v_i^4\in \hc^*(Z_X)$ with $\fe_i\in \hc^*(X)$. For $i=0,1,\hdots,n$, let $$g_i=\diag (\underbrace{1,\hdots,1}_{n-i},\underbrace{-1,\hdots,-1}_{2i},\underbrace{1,\hdots,1}_{n-i})\in G.$$

\begin{thm}\label{SpnSWCs}
The total SWC of an orthogonal representation $\pi$ of $G=\Sp(2n,q)$ is given by  
\begin{equation}
w(\pi)=\prod_{k=1}^n \left( \mc D^{[k]} (\mf e_1,\hdots,\mf e_n) \right)^{m_k(\pi)/4},
\end{equation}
where $m_k(\pi)=\cfrac{1}{2^n}\sum\limits_{i=0}^n\chi_{\sigma_i}(g_k)\chi_\pi(g_i)$.

\end{thm}
Note that the character values $\chi_{\sigma_i}(g_k)$ are given by Lemma \ref{chisig}.

 \begin{ex} For the regular representation $\reg(G)$, we have

$$w(\reg(G))=\mathcal{D}(\fe_1,\hdots,\fe_n)^{|G|/2^{n+2}}.$$ 

\end{ex}

\subsection{Gow's Formula}
In this section, let $G=\Sp(2n,q)$.

\begin{thm}[\cite{gow1985real}, Theorem 1]\label{cen}
 Let $\pi$ be an irreducible self-dual representation of $G$ with central character $\omega_\pi$ and Frobenius-Schur Indicator $\varepsilon(\pi)$. Then, we have
\begin{equation}\label{gowf}
\varepsilon(\pi)=\omega_\pi(-\mathbbm1).
\end{equation}
\end{thm}
We simply call this equality \emph{Gow's formula.}
 For irreducible orthogonal $\pi$ of $G$, it means $\chi_\pi(\mathbbm{1})=\chi_\pi(-\mathbbm{1})$.
 This is same as $\chi_\pi(g_0)=\chi_\pi(g_n)$, which leads to
\begin{equation}\label{e1}\chi_\pi(g_i)=\chi_\pi(g_{n-i})\quad;\quad 0\leq i\leq n.\end{equation}
Now, Lemma \ref{msdeed} simplifies the exponents $m_k(\pi)$ in Theorem \ref{SpnSWCs} when $\pi$ is either irreducible orthogonal, or the symmetrization of an irreducible symplectic:
\begin{cor}\label{Spnio}
Let $\pi$ be an irreducible orthogonal representation of $G$. Then the total SWC of $\pi$ is
\beq
w(\pi)=\prod_{k=1}^n\Big(\mc D^{[k]} (\mf e_1,\hdots,\fe_n)\Big)^{m_{k}(\pi)/4}  
\eeq
where 
\beq
m_{k}(\pi)=\begin{cases}
0 & \text{ when } k \text{ is odd}\\
\cfrac{1}{2^{n-1}}\sum\limits_{i=0}^{\frac {n-1}2}\chi_{\sigma_i}(g_k)\chi_\pi(g_i) & \text{ when } k \text{ is even, } n \text{ is odd}\\[10pt]
\cfrac{1}{2^{n-1}}\sum\limits_{i=0}^{\frac {n-2}2}\chi_{\sigma_i}(g_k)\chi_\pi(g_i)+\cfrac{1}{2^{n}}\big(\chi_{\sigma_{\frac n2}}(g_k)\chi_\pi(g_{\frac n2})\big) & \text{ when } k,n \text{ both are even}.
\end{cases}
\eeq  
\end{cor}

\begin{cor} 
Let $\varphi$ be an irreducible symplectic representation of $G$. Then we have

\beq
w(S(\varphi))=\prod_{k=1}^n\Big(\mc D^{[k]} (\mf e_1,\hdots,\fe_n)\Big)^{m_{k}(\varphi)/2}  
\eeq
where  
\beq
m_{k}(\varphi)=\begin{cases}
0 & \text{ when } k \text{ is even and } n \text{ is odd}\\
\cfrac{1}{2^{n-1}}\sum\limits_{i=0}^{\lceil \frac {n-2}2 \rceil}\chi_{\sigma_i}(g_k)\chi_\varphi(g_i) & \text{ when } k \text{ is odd, }\\[10pt]
  \cfrac{1}{2^{n}} \chi_{\sigma_{\frac n2}}(g_k)\chi_\varphi(g_{\frac n2})  & \text{ when } k,n \text{ both are even}.
\end{cases}
\eeq  
\end{cor}
In the formula above, $\left\lceil\cdot \right\rceil$ denotes the ceiling function; in particular  \[
\left\lceil \frac{n-2}{2} \right\rceil =
\begin{cases}
\frac{n-2}{2}, & \text{if \(n\) is even},\\[6pt]
\frac{n-1}{2}, & \text{if \(n\) is odd}.
\end{cases}
\]
 
\section{Examples} \label{examples.section}

\subsection{The case $n=1$} \label{SL2}
Let $G=\SL(2,q)$, and $\pi$ an orthogonal representation of $G$. Theorem \ref{SpnSWCs} applied for $n=1$ gives $w(\pi)=(1+\fe)^{m_1(\pi)/4}$, where $m_1(\pi)$ in terms of character values at $g_0=\mathbbm1$, $g_1=-\mathbbm1$ is:
 \begin{align*}
m_1(\pi)&=\frac12\big(\chi_\pi(\mathbbm1)-\chi_\pi(-\mathbbm1)\big).
\end{align*}
In particular,
\begin{equation} \label{w4.sl2}
w_4(\pi)=\frac{1}{8}(\deg \pi-\chi_\pi(-\mathbbm 1)) \mf e.
\end{equation}
When $\pi$ is irreducible orthogonal, $m_1(\pi)=0$ by Corollary \ref{Spnio}. Therefore, $w(\pi)$ is trivial for such representations. (Compare \cite[Theorem 1.2]{NSSL2}.)

\subsection{The case $n=2$} 
Let 
$g_0=\mathbbm{1},
 g_1=\diag(1,-1,-1,1),
 g_2=-\mathbbm{1}\in G.$ With $n=2$ in Theorem \ref{SpnSWCs}, the total SWC of an orthogonal representation $\pi$ of $G$ is:
\begin{equation}\label{wSp4}
\begin{split}
w(\pi)&=\Big(\mc D^{[1]}(\fe_1,\fe_2)\Big)^{m_1(\pi)/4}\Big(\mc D^{[2]}(\fe_1,\fe_2)\Big)^{m_2(\pi)/4}\\
&=\Big((1+\mathfrak{e}_1)(1+\mathfrak{e}_2)\Big)^{m_1(\pi)/4}\Big(1+\mathfrak{e}_1+\mathfrak{e}_2\Big)^{m_2(\pi)/4},
\end{split}
\end{equation} where
\begin{align*}
m_1(\pi)&=\frac{1}{4}\big(\chi_{\sigma_0}(g_1)\chi_\pi(\mathbbm 1)+\chi_{\sigma_1}(g_1)\chi_\pi(g_1)+\chi_{\sigma_2}(g_1)\chi_\pi(-\mathbbm 1)\big)\\
&=\frac{1}{4}\big(\chi_\pi(\mathbbm{1})-\chi_\pi(-\mathbbm{1})\big), \text{ and}
\end{align*}
\begin{align*}
m_2(\pi)&=\frac{1}{4}\big(\chi_{\sigma_0}(g_2)\chi_\pi(\mathbbm 1)+\chi_{\sigma_1}(g_2)\chi_\pi(g_1)+\chi_{\sigma_2}(g_2)\chi_\pi(-\mathbbm 1)\big)\\
&=\frac{1}{4}\big(\chi_\pi(\mathbbm{1})-2\chi_\pi(g_1)+\chi_\pi(-\mathbbm{1})\big).
\end{align*}
Here,  the character values $\chi_{\sigma_i}(g_k)$ are obtained using Lemma \ref{chisig}, by expanding the polynomials $(1-y)^k(1+y)^{2-k}$ for $k=1,2$. (This proves Theorem \ref{intro.thm.sp4}.)

Let $\mc E_i=\mc E_i(\mf e_1,\mf e_2)$. From Equation \eqref{wSp4}, we deduce
\begin{equation} \label{w4}
\begin{split}
w_4(\pi) &=\frac{m_1(\pi)+m_2(\pi)}{4}(\fe_1+\fe_2) \\
	&= \frac{1}{8} \left( \deg \pi - \chi_\pi(g_1) \right)\mc E_1  \text{, and}  \\
\end{split}
\end{equation}

 \begin{equation}\label{sp4w8}
 \begin{split}
 w_8(\pi) &= m_1(\pi) \mc E_2 + \left( \binom{m_1(\pi)}{2}+  \binom{m_2(\pi)}{2} \right) \mc E_1^2. \\
\end{split}
 \end{equation}

When $\pi$ is irreducible orthogonal, Corollary \ref{Spnio} leads to the simplification:
$$w(\pi)=\Big(1+\mathfrak{e}_1+\mathfrak{e}_2\Big)^{m_2(\pi)/4}$$
where 
\begin{align*}
m_2(\pi)=\frac{1}{2}\big(\chi_\pi(\mathbbm{1})-\chi_\pi(g_1)\big).
\end{align*}

To prove Corollary \ref{HSW}, we need to describe some representations of $G=\Sp(4,q)$ and find some of their character values, which we will do now.
 
First is a parabolically induced representation given as follows. Let $B$ be the Borel subgroup, consisting of the upper triangular matrices of $G$. Consider the Levi subgroup 
$$M=\left\{\begin{pmatrix}
s&0&0&0\\
0&\ast&\ast&0\\
0&\ast&\ast&0\\
0&0&0&s^{-1}
\end{pmatrix}: s\in \f_q^\times\right\}\cong \f_q^\times\times \mc S.$$

Let $P$ be the parabolic subgroup generated by $M$ and $B$. We have $|P|=q^4(q-1)(q^2-1)$. Let $T_1$ be the subgroup with elements of the form $\diag(s,1,1,s^{-1})$.
Consider a linear character $\chi$ of $\f_q^\times$ with $\chi(-1) = -1$. Set $\alpha=\chi\circ\pr$, where $\pr$ is the projection map $M \twoheadrightarrow T_1$. Then, $\alpha(\diag(-1,1, 1,-1)) =- 1$.

Consider the parabolic induction $\pi_1=\Ind_P^G \as$. We have $\deg \pi=(q+1)(q^2+1)$, $\chi_{\pi_1}(-\mathbbm1)=-(q+1)(q^2+1)$ (as  $\as$ is odd), and $\chi_{\pi_1}(g_1)=0$. (By the Frobenius formula for induced representations.)
 
Next is a representation induced from $X$. Let $\psi$ be an odd linear character of $\f_{q^2}^\times$ with $\psi^q\neq \psi$. 
According to  \cite[Section 5.2]{fulton}, there is an  irreducible cuspidal representation $\sigma_\psi$ of  $\mc S$ corresponding to $\psi$, with degree $q-1$. Viewing $X$ as the product $\mc S \times \mc S$ as in Section \ref{subgroup.sec}, we form the external tensor product representation $\sigma_\psi\boxtimes 1$ of $X$. Finally, let $\pi_2=\Ind_X^G(\sigma_\psi\boxtimes 1)$.
Here, $\deg(\pi_2)=q^2(q-1)(q^2+1)$ and again as $\psi$ is odd, $\chi_{\pi_2}(-\mathbbm1)=-\deg \pi_2.$ Also, $\chi_{\pi_2}(g_1)=0$:  observe that the only conjugates of $g_1$ in $X$ are $\pm g_1$ and there is a Weyl element sending $g_1$ to $-g_1$ upon conjugation. (And again use the Frobenius character formula.)

\begin{proof}[ Proof of Corollary \ref{HSW}]
Let $\Pi_1=S(\pi_1)$, $\Pi_2=S(\pi_2)$ with $\pi_1$, $\pi_2$ as above. Using their character values, we obtain $m_{1}(\Pi_1)=(q+1)(q^2+1)$, $m_{1}(\Pi_2)=q^2(q-1)(q^2+1)$ and $m_{2}(\Pi_2)=m_{2}(\Pi_2)=0$. Note that 
 $m_1(\Pi_1)/4$ is odd when $q\equiv 1$ mod ($4$), and $m_1(\Pi_2)/4$ is odd when $q\equiv 3$ mod ($4$).
 
Corresponding to the cases, we have
\begin{align*}
w(\Pi_i)&=(1+\fe_1+\fe_2+\fe_1\fe_2)^{m_{1}(\Pi_i)/4}\\
&=1+\frac{m_{1}(\Pi_i)}{4}\left(\fe_1+\fe_2+\fe_1\fe_2\right)+\hdots
\end{align*} hence $\mc E_1,\mc E_2\in \hc^*_{\SW}(G)$.

\end{proof}

\subsection{The case $n=4$} 

Let $\ul \fe=(\fe_1,\fe_2,\fe_3,\fe_4)$.
 We apply Theorem \ref{SpnSWCs} for $n=4$ to have the total SWC of an orthogonal representation $\pi$ of $G$:
\enlargethispage{\baselineskip}
\begin{equation}\label{Sp6}
 w(\pi)=\mc D^{[1]}(\ul \fe)^{m_1(\pi)/4}\mc D^{[2]}(\ul \fe)^{m_2(\pi)/4}
 \mc D^{[3]}(\ul \fe)^{m_3(\pi)/4}\mc D^{[4]}(\ul \fe)^{m_4(\pi)/4};
 \end{equation}
please see Section \ref{dickin} for calculation of these $\mc D^{[i]}(\ul \fe)$.

Also, Gow's formula through Corollary \ref{Spnio} allows simplification for irreducible orthogonal $\pi$:
$$w(\pi)=\Big(1+ \mc E_1+ \mc E_1^2 + \mc E_1^3+ \mc E_2^2+\mc E_1\mc E_3+\mc E_1\mc E_2^2+\mc E_1^2\mc E_3+\mc E_1\mc E_2\mc E_3+\mc E_3^2+\mc E_1^2\mc E_4\Big)^{m_2(\pi)/4}(1+\mc E_1)^{m_4(\pi)/4}$$
where, with the help of Lemma \ref{chisig} once again, gives \beq 
\begin{split}
m_2(\pi)&=\cfrac{1}{8}\big(\deg(\pi)-\chi_\pi(g_2)\big),\\
m_4(\pi)&=\cfrac{1}{8}\big(\deg(\pi)-4\chi_\pi(g_1)+3\chi_\pi(g_2)\big).
\end{split}
\eeq

Again, $g_1=\diag(1,1,1,-1,-1,1,1,1)$ and
 $ g_2=\diag(1,1,-1,-1,-1,-1,1,1)$.

\subsection{Weil Representations} \label{weil.section} 
For a fixed nontrivial linear character $\lambda$ of $(\f_q,+)$, the prescription
$$ x  \mapsto \lambda_x\quad;\quad \lambda_x(y)=\lambda(xy)$$
 defines a group isomorphism  $\f_q \overset{\sim}{\to} \widehat{\f}_q$.
Let $\as$ be a non-square, and $1$ be the identity in $\f_q^\times$.  According to G\'{e}rardin \cite[Theorem 2.4(d)]{PGWeil}, the symplectic group $G=\Sp(2n,q)$ has two nonisomorphic \emph{Weil representations} $\mathcal{W}_{\lambda_1}$, $\mathcal{W}_{\lambda_\as}$ associated to the characters $\lambda_1$, $\lambda_\as$. Set 
$\mathcal{W}=\mathcal{W}_{\lambda_1}$ and $\mathcal{W}'=\mathcal{W}_{\lambda_\as}$.  Both $\mathcal{W}$, $\mathcal{W}'$ are complex representations of degree $q^n$, and decompose into two irreducible representations as
$$\mathcal{W}=\mathcal{W}_{l}\oplus\mathcal{W}_{s}\quad, \quad \mathcal{W}'=\mathcal{W}'_{l}\oplus\mathcal{W}'_{s} $$
where $\mathcal{W}_{l}$, $\mathcal{W}'_l$ are the components with larger degrees $(q^n+1)/2$, and $\mathcal{W}_{s}$, $\mathcal{W}'_{s}$ have degrees $(q^n-1)/2$. When $q\equiv 3$ (mod $4$), none of these four irreducible representations are self-dual. In fact, $\mathcal{W}_{l}$, $\mathcal{W}'_{l}$ are dual to each other, and similarly $\mathcal{W}_{s}$ is the dual of $\mathcal{W}'_s$. Whereas when $q\equiv 1$ (mod $4$), all representations of $G$ are self-dual. In this case, exactly one of $\mathcal{W}_{l}$, $\mathcal{W}_{s}$ is orthogonal, and the other is symplectic. The same is true for the components of $\mathcal{W}'$.  So, neither $\mathcal{W}$ nor $\mathcal{W}'$ is orthogonal. Here we compute 
\beq
w(S(\mathcal{W}))=c(\mc W) \mod 2.
\eeq

From\cite[Corollary 4.8.1]{PGWeil}, the character values of $\mathcal W$ at $g_i$ are given by:
$$\chi_{\mathcal W}(g_i)=(-1)^{i\frac{q-1}{2}}q^{N(g_i)},  \text{ with }\quad N(g)=\frac12\dim_{\mb F_q}\ker(g-\mathbbm1).$$
One sees that
$$\chi_{\mathcal W}(g_i)=(-1)^{i\frac{q-1}{2}}q^{n-i}.$$
For each $1\leq k\leq n$, Equation \eqref{mk} gives
\begin{align*}m_{k}(S(\mathcal{W}))&=\frac{1}{2^{n-1}}\sum_{i=0}^n [f_k]_i \chi_\w(g_i)\\
&=\frac{1}{2^{n-1}}\sum_{i=0}^n [f_k]_i (-1)^{i\frac{q-1}{2}}q^{n-i}\end{align*}
where $[f_k]_i$ is the coefficient of $x^i$ in $f_k(x)=(1-x)^k(1+x)^{n-k}$. Note that  $[f_k]_{n-i}=(-1)^k[f_k]_i$ for each $i=0,\hdots,n$, so that
$$m_{k}(S(\mathcal{W}))=\frac{1}{2^{n-1}}\sum_{i=0}^n (-1)^k [f_k]_{n-i} (-1)^{i\frac{q-1}{2}}q^{n-i}.$$
Put $F_k(x):=(x-1)^k(x+1)^{n-k}=(-1)^kf_k(x)$. Then, for $q\equiv 1$ (mod $4$), it is clear that
$$m_{k}(S(\mathcal{W}))=\frac{1}{2^{n-1}}F_k(q)=\frac{1}{2^{n-1}}(q-1)^k(q+1)^{n-k}.$$
Whereas for $q\equiv 3$ (mod $4$), with a few manipulations, we get 
\begin{align*}m_{k}(S(\mathcal{W}))&=\frac{1}{2^{n-1}}\sum_{i=0}^n  (-1)^{i}[F_k]_{n-i} q^{n-i}\\
&=\frac{1}{2^{n-1}}\sum_{j=0}^n  (-1)^{n-j}[F_k]_{j} q^{j}\\
&=\frac{(-1)^n}{2^{n-1}}F_k(-q)\\
&=\frac{1}{2^{n-1}}(q-1)^{n-k}(q+1)^k.
\end{align*}
\begin{cor}
For the Weil representations $\w$ and $\w'$ of $G$, we have

$$w(S(\w))=w(S(\w'))=\prod_{k=1}^n\mc D^{[k]} (\mf e_1, \ldots, \mf e_n)^{m_{k}(S(\w))/4}$$
where $$m_{k}(S(\w))=\begin{cases}
\cfrac{1}{2^{n-1}}(q-1)^{k}(q+1)^{n-k},& q\equiv 1\;(\Mod 4)\\
\cfrac{1}{2^{n-1}}(q-1)^{n-k}(q+1)^{k},& q\equiv 3\;(\Mod 4).
\end{cases}$$
\end{cor}

\begin{proof}
We have already shown this for $\w$. In \cite{PGWeil}, we see that $\chi_\w(g_i)=\chi_{\w'}(g_i)$ for each $i$. Thus, the formula  also holds for $w(S(\w'))$.
\end{proof}

\section{Universal SWCs}    \label{universal.section}   
Let us write $X_n$ for the subgroup of $G_n=\Sp(2n,q)$  isomorphic to $\mc S^n$ (previously just written as $X$).
The restriction map $\hc^*(G_n)$ to $\hc^*(X_n)$ has image in the $S_n$-invariants $\hc^*(X_n)^{S_n}$, which  contains the elements
$\mc E_j=\mc E_j(\fe_1, \ldots, \fe_n)$ but also
 \beq
\mc F_j=\sum_{\substack{i_1<\hdots< i_j\\1\leq k\leq j}}\fe_{i_1}\hdots\hat{\fe}_{i_k}\hdots\fe_{i_j}b_{i_k}.
 \eeq
Note that $ \deg (\mc E_j)=4j$ and $\deg(\mc F_j)=4j-1$.

\begin{thm}\label{lind} Let $E_1, \ldots, E_n, F_1, \ldots, F_n$ be formal indeterminates. The homomorphism 
\beq
\mathbb F_2[E_1, \ldots, E_n, F_1, \ldots, F_n] \to \hc^*(X_n)^{S_n}
\eeq
 defined by sending $E_i$ to $\mc E_i$ and $F_i$ to $\mc F_j$ is surjective with kernel equal to the ideal $(F_1^2, \ldots, F_n^2)$.
\end{thm}

\begin{proof} This is implicit in \cite[Theorem 6.1, page 283]{FP}.
\end{proof}

Let $m\geq n$ and $\iota_n: G_{n-1}\to G_m$ be the following inclusion: For $A\in G_{n-1}$,
$$\iota_n(A)=\begin{pmatrix}

\mathbbm 1_{m-n+1}&&\\
&A&\\
&&\mathbbm 1_{m-n+1}
\end{pmatrix}.$$

\begin{proof}[Proof of Theorem \ref{intro.thm.univers}]
By Lemma \ref{spsl2}, it is enough to show that the restriction  \newline $\hc^i(X_m)^{S_m}\to \hc^i(X_{n-1})^{S_{n-1}}$ is injective. Write $R_n=P_n\otimes_{\mathbb F_2} Q_n$, where
 $P_n=\mathbb F_2[\mc E_1,\hdots, \mc E_n]$ is the subalgebra generated by $\mc E_1,\hdots, \mc E_n$ and $Q_n=\mathbb F_2[\mc F_1,\hdots, \mc F_n]$ is the subalgebra generated by $\mc F_1,\hdots, \mc F_n$ in $\hc^*(\mc S^n)$. From above, we need to show that the restriction of the map $P_m\otimes Q_m\to P_{n-1}\otimes Q_{n-1}$ to the degree $k$ part is injective for all $k<4n-1$.

Any monomial of degree $k$ is of the form $\mc E_1^{r_1}\hdots \mc E_m^{r_m}\mc F_1^{s_1}\hdots  \mc F_m^{s_m}$ with $r_i\geq 0$, $s_i\in\{0,1\}$. We may write this as $\mc E^{\mathbf r} \mc F^{\mathbf s}$, with  $\mathbf r=(r_1,\hdots,r_{m})$ and $\mathbf s=(s_1,\hdots,s_{m})$. By Theorem \ref{lind} above, these monomials are linearly independent.
 The condition 
\beq
\sum_{i=1}^m (4i)r_i+\sum_{j=1}^m (4j-1)s_j=k<4n-1
\eeq
ensures that if $i\geq n$, then $r_i=0$ and $s_i=0$.

An element $\as=\sum_{(\mathbf r,\mathbf s)}c_{\mathbf r,\mathbf s}\mc E^{\mathbf r} \mc F^{\mathbf s}\in \hc^{k}(G_m)$ under $\iota^*_n$ gets mapped to 
$$\sum_{\substack{\mathbf r=(r_1,\hdots,r_{n-1})\\\mathbf s=(s_1,\hdots,s_{n-1})}}c_{\mathbf r,\mathbf s}\mc E^{\mathbf r}\mc F^{\mathbf s}\in \hc^{k}(G_{n-1}).$$
If $i^*_n(\as)=0$, then all $c_{\mathbf{r},\mathbf{s}}$ are zero since $\mc E^{\mathbf r}\mc F^{\mathbf s}$ are linearly independent. This implies $\as=0$, and hence we have the injectivity of $\iota^*_n$ for $k<4n-1$.
\end{proof}

\begin{proof}[Proof of Theorem \ref{univ.intro}]
By Theorem \ref{intro.thm.univers}, the map  $i^*_2:\hc^4(G_m)\to \hc^4(G_1)=\hc^4(\mc S)$ is injective for all $m\geq 1$. So the formula for $w_4(\pi)$ follows from \eqref{w4.sl2}. Similarly, since $i^*_3:\hc^8(G_m)\to \hc^8(G_2)$ is injective for all $m\geq 3$, we have $w_8(\pi)$  from Equation \eqref{sp4w8}.
\end{proof}
 
 \appendix
\section{}

In this Appendix we discuss the formula for CCs and SWCs of external tensor products of representations. This formula involves an interesting family of polynomials which we now define.

\subsection{The Polynomials $\PP_{m,n}$}

Given non-negative integers $m,n$, we define $q_{m,n}$ be the following polynomial in $\Lambda=\z[x_1,\hdots,x_m,y_1,\hdots,y_n]$:
\begin{equation}\label{qmn}q_{m,n}(\mathbf{x},\mathbf{y})=\prod_{i=1}^m\prod_{j=1}^n(1+x_i+y_j).\end{equation}
It is symmetric in the $x_i$ and the $y_j$ separately, so by the fundamental theorem of symmetric polynomials \cite[Chapter 4, Section 6, Theorem 1]{b47}, there is a unique polynomial $\PP_{m,n}\in\Lambda$ such that
\begin{equation}\label{polyPq}
q_{m,n}(\mathbf{x},\mathbf{y})=\PP_{m,n}(\mc E_1(\mathbf {x}),\hdots,\mc E_m(\mathbf{x}),\mc E_1(\mathbf y),\hdots,\mc E_n(\mathbf y)).
\end{equation}

For example, we have  $ \PP_{m,0}({\bf x},{\bf y})= 1+ \sum_{i=1}^m  x_i$.

  \subsection{Characteristic Classes of External Tensor Products} 
  Let $G_1,G_2$ be finite groups, and let $p_i:G_1\times G_2\to G_i$ be the projection maps.  Given complex $G_i$-representations  $(\pi_i,V_i)$, we consider their external tensor product $(\pi_1\boxtimes\pi_2, V_1\otimes V_2)$.

\begin{prop}\label{tenswc}

Suppose $(\pi_1,V_1),$ $(\pi_2,V_2)$ have respective degrees $m,n$. Then,
$$c(\pi_1\boxtimes\pi_2)=\PP_{m,n}(p_1^*c_1(\pi_1),\hdots,p_1^*c_m(\pi_1),p_2^*c_1(\pi_2),\hdots,p_2^*c_n(\pi_2)),$$ 
where $\PP_{m,n}$ is as defined above. When $\pi_1,\pi_2$ are orthogonal, then
$$w(\pi_1\boxtimes\pi_2)=\PP_{m,n}(p_1^*w_1(\pi_1),\hdots,p_1^*w_m(\pi_1),p_2^*w_1(\pi_2),\hdots,p_2^*w_n(\pi_2)).$$ 
\end{prop}

We will sketch the proof for SWCs in this Appendix; the case of CCs is similar. 
Iterating Proposition \ref{tenswc} gives:

\begin{cor}\label{tengen}
Let $\Pi=\pi_1\boxtimes\cdots\boxtimes \pi_n$ be a (complex) representation of $G=G_1\times\cdots\times G_n$. Then $c_k(\Pi)$ is in the subalgebra of $\hc^*(G,\z)$ generated by $\{p_j^*c_i(\pi_j): 1\leq j\leq n, i\geq 0 \}$. When each $\pi_i$ is orthogonal, the class
 $w_k(\Pi)$ is in the subalgebra of $\hc^*(G)$ generated by $\{p_j^*w_i(\pi_j): 1\leq j\leq n, i\geq 0 \}$.
\end{cor}

\subsection{Products of Vector Bundles}

The results in the previous section are valid for vector bundles over paracompact spaces. Here, we will prove them in this more general setting.
All our base spaces are paracompact.

 Let $E_1$, $E_2$ be two real vector bundles over a base space $B$. Their internal tensor product $E_1\otimes E_2$  is again a real vector bundle over $B$. 
 
\begin{prop}[\cite{milnor}, Chapter 7, Problem 7-C]\label{tench}
Let $E_1$, $E_2$ be real vector bundles over $B$ with respective ranks $m,n$. Then,
$$w(E_1\otimes E_2)=\PP_{m,n}(w_1(E_1),\hdots,w_m(E_1),w_1(E_2),\hdots,w_n(E_2)).$$

\end{prop}

Let $E_1$, $E_2$ be real vector bundles over base spaces $B_1$, $B_2$ respectively with $\Pi_i:E_i\to B_i$. We can form their external tensor product $E_1\boxtimes E_2$ over $B_1\times B_2$ as follows.

Consider the projection maps $p_i: B_1\times B_2\to B_i$. Let $p_i^*E_i$ be the pullback of $E_i$ by $p_i$ consisting of elements $((b_1,b_2),e_i)\in (B_1\times B_2)\times E_i$ such that $\Pi_i(e_i)=p_i(b_1,b_2)=b_i$. These are vector bundles over the same base space $B_1\times B_2$. Thus, we construct their internal tensor product, and put $E_1\boxtimes E_2:=p_1^*E_1\otimes p_2^*E_2$, which is again a vector bundle over $B_1\times B_2.$
This with the naturality of SWCs and Proposition \ref{tench} gives:
\begin{prop}\label{tenswc1}
Let $E_1$, $E_2$ be real vector bundles over respective base spaces $B_1$, $B_2$ with $\rank(E_1)=m$ and  $\rank(E_2)=n$. Then,
$$w(E_1\boxtimes E_2)=\PP_{m,n}(p_1^*w_1(E_1),\hdots,p_1^*w_m(E_1),p_2^*w_1(E_2),\hdots,p_2^*w_n(E_2)).$$

\end{prop}

\subsection{Vector Bundles From Representations}\label{ETPs}

Let $G$ be a finite group, and $(\rho, U)$ be a real representation of $G$. Associated to $G$ is a classifying space $BG$ with a contractible right principal  $G$-bundle $EG$. From $(\rho,U)$ one can form the associated real vector bundle $EG[U]$ over $BG$. To define $EG[U]$, first form the product $EG \times U$; it is $G$-space under the action $g \cdot (x,u)=(x \cdot g^{-1}, \rho(g)u)$. Then $EG[U]$ is the quotient of $EG \times U$ by this action. Then put
$w_i^\R(\rho)=w_i(EG[U])$.
(See for instance \cite{Benson} or  \cite{GKT}.)  The singular cohomology $\hc^*(BG, A)$ is naturally isomorphic to the group cohomology $\hc^*(G, A)$ for any abelian group  $A$.

 Let $(\pi,V)$ be a complex orthogonal representation of $G$. There is a representation $(\rho, U)$, with $U$ a real vector space, so that $\rho \otimes_{\rr} \cc \cong \pi$ (see \cite[Chapter II, Section 6]{BrokerDieck} for instance). Such a representation is called a \emph{real form} of $\pi$; it is unique up to isomorphism. We now can define $w(\pi):=w^\R(\rho)$.

Let $G_1,G_2$ be finite groups with orthogonal complex representations  $(\pi_1,V_1)$ and $(\pi_2,V_2)$. Let $G=G_1 \times G_2$. We can form their external tensor product $(\pi_1\boxtimes\pi_2, V_1\otimes V_2)$, an orthogonal representation of $G$. Let $(\rho_i,U_i)$ be real forms of $\pi_i$ for $i=1,2$. Then $\rho_1 \boxtimes \rho_2$ is a real form of $\pi_1 \boxtimes \pi_2$. From above, one has 
\beq
\begin{split}
w(\pi_1\boxtimes\pi_2) &= w^{\R}(\rho_1 \boxtimes \rho_2) \\
					&= w(EG[U_1 \otimes U_2]) \\
				&= w(\mc U_1\boxtimes\mc U_2) \in \hc^*(BG_1 \times BG_2,\mathbb F_2), \\
				\end{split}
\eeq

 where
$\mc U_i=EG_i[U_i]$. 
We may identify $BG_1 \times BG_2$ with $BG$. Now, Proposition \ref{tenswc} for SWCs   follows from Proposition \ref{tenswc1} with $E_i=\mc U_i$.

 \bigskip

\subsection*{Acknowledgments.}
 Part of this paper comes out of the first author’s Ph.D. thesis \cite{Malik.thesis} at IISER Pune, during which she was supported by a Ph.D. fellowship from the Council of Scientific and Industrial Research, India. Both authors thank the Chennai Mathematical Institute for its hospitality during many visits.
\bibliographystyle{abbrv}
\bibliography{mybib}

@article {GKT,
    AUTHOR = {Gunarwardena, J. and Kahn, B. and Thomas, C.},
     TITLE = {Stiefel-{W}hitney classes of real representations of finite
              groups},
   JOURNAL = {J. Algebra},
  FJOURNAL = {Journal of Algebra},
    VOLUME = {126},
      YEAR = {1989},
    NUMBER = {2},
     PAGES = {327--347},
      ISSN = {0021-8693},
   MRCLASS = {20J06 (20C99 55R40)},
  MRNUMBER = {1024996},
MRREVIEWER = {V. P. Snaith},
       DOI = {10.1016/0021-8693(89)90309-8},
       URL = {https://doi-org.proxy.queensu.ca/10.1016/0021-8693(89)90309-8},
}

@book{BrokerDieck,
	Author = {Br\"ocker, T. and Dieck, T.T.},
	Isbn = {0-387-13678-9},
	Mrclass = {22E45 (22-01 57-01)},
	Mrnumber = {1410059},
	Pages = {x+313},
	Publisher = {Springer-Verlag, New York},
	Series = {Graduate Texts in Mathematics},
	Title = {Representations of compact {L}ie groups},
	Volume = {98},
	Year = {1995}}

@book{milnor,
	Author = {Milnor, J. and Stasheff, J.D.},
	Publisher = {Princeton university press},
	Title = {Characteristic Classes.(AM-76)},
	Volume = {76},
	Year = {2016}}

@book {Milgram,
    AUTHOR = {Adem, A. and Milgram, R.J.},
     TITLE = {Cohomology of finite groups},
    SERIES = {Grundlehren der Mathematischen Wissenschaften $($Fundamental
              Principles of Mathematical Sciences$)$},
    VOLUME = {309},
   EDITION = {Second},
 PUBLISHER = {Springer-Verlag, Berlin},
      YEAR = {2004},
     PAGES = {viii+324},
      ISBN = {3-540-20283-8},
   MRCLASS = {20J06 (18G99 55R35 55R40)},
  MRNUMBER = {2035696},
       DOI = {10.1007/978-3-662-06280-7},
       URL = {https://doi-org.proxy.queensu.ca/10.1007/978-3-662-06280-7},
}

@article{KT,
	Author = {Kamber, F. and Tondeur, Ph.},
	Doi = {10.2307/2373408},
	Fjournal = {American Journal of Mathematics},
	Issn = {0002-9327},
	Journal = {Amer. J. Math.},
	Mrclass = {57.30},
	Mrnumber = {0220309},
	Mrreviewer = {K. H. Mayer},
	Pages = {857--886},
	Title = {Flat bundles and characteristic classes of group-representations},
	Url = {https://doi.org/10.2307/2373408},
	Volume = {89},
	Year = {1967}}

@book {Benson,
    AUTHOR = {Benson, D. J.},
     TITLE = {Representations and cohomology. {II}},
    SERIES = {Cambridge Studies in Advanced Mathematics},
    VOLUME = {31},
   EDITION = {Second},
 PUBLISHER = {Cambridge University Press, Cambridge},
      YEAR = {1998},
     PAGES = {xii+279},
      ISBN = {0-521-63652-3},
   MRCLASS = {20-02 (19-02 20Cxx 20Jxx 55-02)},
  MRNUMBER = {1634407},
}

@book{fulton,
  title={Representation theory: a first course},
  author={Fulton, W. and Harris, J.},
  volume={129},
  year={2013},
  publisher={Springer Science \& Business Media}
}

@book{FP,
  title={Homology of classical groups over finite fields and their associated infinite loop spaces},
  author={Fiedorowicz, Z. and Priddy, S.},
  volume={674},
  year={2006},
  publisher={Springer}
}

@article{wilk,
  title={A primer on the {D}ickson invariants},
  author={Wilkerson, C.},
  journal={Amer. Math. Soc. Contemp. Math. Series},
  volume={19},
  pages={421--434},
  year={1983}
}

@article{gow1985real,
  title={Real representations of the finite orthogonal and symplectic groups of odd characteristic},
  author={Gow, R.},
  journal={Journal of Algebra},
  volume={96},
  number={1},
  pages={249--274},
  year={1985},
  publisher={Academic Press}
}

@article{goren,
  title={Finite groups},
  author={Gorenstein, D.},
  journal={Gorenstein. NY: Harper and Row},
  year={1968}
}

@article{ganguly2022stiefel,
  title={Stiefel {W}hitney classes for real representations of $\text{GL}_2 (\mathbb{F}_q)$},
  author={Ganguly, J. and Joshi, R.},
  journal={International Journal of Mathematics},
  pages={2250010},
  year={2022},
  publisher={World Scientific}
}

@book{may,
  title={A concise course in algebraic topology},
  author={May, J.P.},
  year={1999},
  publisher={University of Chicago press}
}

@article{GJgln,
  title={Total {S}tiefel-{W}hitney classes for real representations of $\text{GL}_n $ over $ \mathbb{F}_q,\mathbb{R} $ and $\mathbb{C} $},
  author={Ganguly, J. and Joshi, R.},
  journal={Research in the Mathematical Sciences },
  year={to appear}
}

@article{NSSL2,
  title={Stiefel-{W}hitney {C}lasses of representations of {SL}($2,q$)},
  author={Malik, Neha and Spallone, Steven},
  journal={Journal of Group Theory},
  volume={26},
  number={5},
  pages={891--914},
  year={2023},
  publisher={De Gruyter}
}

@article{PGWeil,
  title={Weil representations associated to finite fields},
  author={G{\'e}rardin, P.},
  journal={Journal of Algebra},
  volume={46},
  number={1},
  pages={54--101},
  year={1977},
  publisher={Elsevier}
}

@article{NSn,
  title={Decomposition theorem for homology groups of symmetric groups},
  author={Nakaoka, M.},
  journal={Annals of Mathematics},
  pages={16--42},
  year={1960},
  publisher={JSTOR}
}

@phdthesis{Malik.thesis,
	address = {Pune},
	author = {N. Malik},
	school = {Indian Institute of Science Education and Research},
	title = {Stiefel-Whitney Classes of Representations of Some Finite Groups of Lie Type},
	year = {2022}}

@book{b47,
  title={Algebra II: Chapters 4-7},
  author={Bourbaki, N.},
  year={2013},
  publisher={Springer Science \& Business Media}
}

@article{MSSLn,
  title={Stiefel-{W}hitney {C}lasses for finite special linear groups of even rank},
  author={Malik, Neha and Spallone, Steven},
  journal={Journal of Algebra},
  volume={673},
  pages={455--473},
  year={2025},
  publisher={Elsevier}
}

@book{shl,
  title={Quadratic and Hermitian forms},
  author={Scharlau, W.},
  volume={270},
  year={2012},
  publisher={Springer Science \& Business Media}
}
\vspace{10mm}

\end{document}